\documentclass[a4paper,11pt]{amsart}
\newcommand{\query}[1]%
{\mbox{}\marginpar{\raggedright\hspace{0pt}{\small\em #1}}}%
\usepackage{amssymb}
\usepackage{amsmath}
\usepackage{amsfonts}
\usepackage{mathrsfs}
\usepackage[all]{xy}

\usepackage{hyperref}

\theoremstyle{plain}
\newtheorem{theorem}{Theorem}[section]

\newtheorem{lemma}[theorem]{Lemma}
\newtheorem{proposition}[theorem]{Proposition}

\newtheorem*{theorem*}{Theorem}

\theoremstyle{definition}
\newtheorem{definition}[theorem]{Definition}

\newtheorem{example}[theorem]{Example}
\newtheorem{notation}[theorem]{Notation}

\theoremstyle{remark}
\newtheorem{remarque}[theorem]{Remark}

\def\ZZ{{\mathbb Z}}

\def\QQ{{\mathbb Q}}
\def\CC{{\mathbb C}}

\def\D{{\mathfrak{D}}}
\def\EE{{\mathscr{E}}}
\def\TT{{\mathbb{T}}}
\def\G{{\mathbb{G}}}
\def\rank{{\rm rank}}
\def\dim{{\rm dim}}
\def\deg{{\rm deg}}
\def\PP{{\rm PPDiv_{\QQ}}}
\def\tail{{\Sigma_{\star}}}
\def\codim{{\rm codim}}
\def\face{{\rm face}}
\def\supp{{\rm supp}}
\def\hc{{\mathscr{C}}}
\def\val{{\rm val}}
\def\ord{{\rm ord}}
\def\P{{\mathbb{P}}}
\def\AA{{\mathbb{A}}}
\def\spec{{\rm Spec}}
\def\big{{\TT\times\G_{m}}}
\def\IH{{ IH}}
\def\L{{\mathscr{L}}}
\def\germ{{SH(\EE)}}
\def\germh{{SH(\EE)_{\rm hor}}}
\def\germv{{SH(\EE)_{\rm ver}}}
\def\bfan{{\tail^{\rm big}}}
\def\RR{{\mathscr{R}}}
\def\der{{D^b(X;\QQ)}}

\author{Marta Agustin Vicente}
\address{Basque Center of Applied Mathematics (BCAM)\\ Bilbao, Spain}
\email{martaav22@gmail.com} 
\author{Kevin Langlois}
\address{Heinrich Heine Universit\"at  \\ D\"usseldorf, Germany}
\email{langlois.kevin18@gmail.com} 
\title[ On intersection cohomology with torus actions of complexity one]
{On intersection cohomology with torus \\
actions of complexity one}

\begin{document}
\begin{abstract}
The purpose of this article is to investigate the intersection cohomology for algebraic varieties with torus action. Given
an algebraic torus $\TT$, one of our result determines
the intersection cohomology Betti numbers of any normal projective $\mathbb{T}$-variety admitting an algebraic curve as global quotient. The calculation is expressed in terms of a combinatorial description involving a divisorial fan which is 
the analogous of the defining fan of a toric variety. Our main tool to obtain this computation is a description of the decomposition theorem in this context.
\end{abstract}  
\thanks{\em 2010 Mathematics Subject Classification \rm 32S60, 14L30,  	14M25. 
}

\maketitle

\setcounter{tocdepth}{1}
\tableofcontents
\section*{Introduction}
In this article, we are interested in the rational intersection cohomology for complex normal algebraic varieties
endowed with an action of an algebraic torus $\TT$. 

Let us recall that a \emph{$\TT$-variety} is a normal variety endowed with an effective $\TT$-action. 
The \emph{complexity} of a $\TT$-variety $X$
is the non-negative number $\dim(X) - \dim(\TT)$; it corresponds to the transcendence degree over $\CC$ of the field extension $\CC(X)^{\TT}$ of invariant rational functions on $X$. By a result of Rosenlicht (see \cite{Ros63}), the complexity of $X$ is the codimension of a general $\TT$-orbit.
Moreover, if the $\TT$-variety $X$ is of complexity one, then the inclusion
$$\CC(X)^{\TT}\subseteq \CC(X)$$
yields a map
$$\gamma: X\dashrightarrow Y$$
to a smooth projective curve $Y$. Here $Y$ is the projective curve obtained from the algebraic function field $\CC(X)^{\TT}$ of one variable. The map $\gamma$ is called the \emph{rational quotient} of $X$. 
We say that the complexity-one $\TT$-variety $X$ is \emph{contraction-free} if the rational quotient map
given by the $\TT$-action on $X$ is a regular morphism. 

The best concrete examples of $\TT$-varieties are \emph{toric varieties}
(corresponding to complexity zero) and affine $\G_m$-surfaces. Both admit a combinatorial description.
The first one can be defined by a fan of strongly convex polyhedral cones in the rational vector space $N_{\QQ}$ associated with the lattice $N$ of one-parameter
subgroups of the torus $\TT$, see for instance \cite{Ful93, CLS11}. The second one can be described by a $\QQ$-divisor or a pair of $\QQ$-divisors 
on a smooth algebraic curve (see \cite{FZ03} for details). Note that contraction-free $\TT$-varieties of complexity one were studied by Mumford in \cite[Chapter IV]{KKMS73}.
These combinatorial descriptions admit a generalization to the setting of $\TT$-varieties (see \cite{AH06, AHS08, Tim08, Lan14}).
The description in \cite{AH06} of an affine $\TT$-variety is in term of a divisor on a normal variety where its coefficients are polyhedra in $N_{\QQ}$. Such a combinatorial object is called a \emph{polyhedral divisor}. More generally, the combinatorial description introduced in \cite{AHS08} for a $\TT$-variety involves a \emph{divisorial
fan} which corresponds to a finite set of polyhedral divisors (with some additional conditions).

One of the results of this article is an explicit description (in terms of divisorial fans) of the intersection cohomology Betti numbers of projective 
contraction-free $\TT$-varieties of complexity one  (see Theorem \ref{main}).
As an intermediate step, we compute the classical Betti numbers for every smooth projective $\TT$-variety of complexity one
(see Propositions \ref{proposition} and \ref{proposition2}). Proposition \ref{proposition2} was also obtained independently in \cite[Section 2]{LLM16}. These results can be related to classical ones in the field of intersection cohomology with a torus action. See \cite{Kir88} for a general description in the projective case using the Bialynicki-Birula decomposition. See also \cite{Sta87, DL91, Fie91, BBKK99} for the toric
projective case which is related to $h$-polynomials and \cite{FK86} for the case of affine $\G_{m}$-surfaces.

To motivate our result, let us make some comments on the description in \cite{Sta87} for the toric projective case. In the sequel,
for every algebraic variety $V$ we will denote by $P_{V}(t)\in\ZZ[t]$ the Poincar\'e polynomial of $V$ which is
the generating function of the intersection cohomology Betti numbers of $V$ (see \ref{betti1}). Let $X$ be a projective 
toric variety (for the torus $\TT$) with defining fan $\Sigma_{X}$. Then, since $X$ is projective, the fan $\Sigma_{X}$
is the normal fan of a rational polytope $Q$. In particular, the set of faces of dimension $i$
of $Q$ is in bijection with the set of cones of codimension $i$. In \cite{Sta87} a polynomial $h(\Lambda;t)$ (called \emph{$h$-polynomial})
depending on a polytope $\Lambda$ is introduced, so that we have the equality $P_{X}(t) = h(Q;t^{2})$ (see \cite[Theorem 3.1]{Sta87}). 
In the smooth case, the polynomial $h(Q;t^{2})$ can be defined by the relation 
$$h(Q; t^{2}) = \sum_{i = 0}^{n}f_{i}(Q)(t^{2}-1)^{i},$$
where $f_{i}(Q)$ is the number of faces of dimension $i$. However, in the non-smooth case, the polynomial above on the right-hand side
has generally negative coefficients and therefore the definition of an $h$-polynomial is different (see \cite[Section 2]{Sta87} 
and also the remainder in \ref{hpol}). Our main result is an adaptation of the description of the $h$-polynomials explained above (see \cite[Theorem 3.1]{Sta87}) to the setting of torus actions of complexity one.

Let us introduce some notation in order to explain our result. Let $\EE$ be a divisorial fan on a smooth projective curve $Y$ corresponding
to a singular projective contraction-free $\TT$-variety $X(\EE)$ of complexity one. We recall the definition of the divisorial fan $\EE$
and the construction of the variety $X(\EE)$ in \ref{divisorial}. Note that the curve $Y$ is the quotient of the $\TT$-action on $X(\EE)$.
Let us denote by $\supp(\EE)$ the support of $\EE$ which corresponds to points $y\in Y$ where the fiber of the quotient map is non-trivial (see \ref{orbit} for a precise definition). 
Then, similarly to the toric case, one can attach rational polytopes $Q(\EE)$ and $Q_{y}(\EE)$ for every $y\in\supp(\EE)$, see \ref{polytopes} for the construction of these objects. 
Our result can be stated as follows.  

\begin{theorem}
\label{main1}  
Let $g$ be the genus of the curve $Y$ and let $r$ be the cardinality of the finite set $\supp(\EE)$.
Then we have the equality
$$P_{X(\EE)}(t) =  ((1-r)t^{2} + 2gt + 1- r)h(Q(\EE); t^{2}) + \sum_{y\in \supp(\EE)}h(Q_{y}(\EE); t^{2}).$$
\end{theorem} 

In particular, one can see that when $X(\EE)$ is a rational variety, the odd rational intersection cohomology groups of $X(\EE)$ vanish as in
the toric projective case.    

To obtain this result, we adapt to our setting a version of the decomposition theorem in \cite{CMM15} given for toric fibrations (see
Proposition \ref{decomp}). This allows us to show the result by induction on the dimension of $X(\EE)$ starting with a projective desingularization of $X(\EE)$ given by a subdivision
of divisorial fans (see \ref{subf} for the definition of a subdivision of divisorial fans). 

Our version of the decomposition theorem is expressed in terms of families of natural numbers called \emph{s-sequences} which are related to
the topology of toric proper maps. The subvarieties appearing in the decomposition theorem are all $\TT$-stable. They are parametrized by a set 
$\germ$ depending combinatorially on the divisorial fan $\EE$. For an algebraic variety $X$ we will denote by $IC_{X}$ its intersection cohomology complex with rational coefficients.
Our result (see Proposition \ref{decomp})  can be enunciated as follows.
\begin{theorem}
\label{main2} 
Let $\EE$ be a divisorial fan on $(Y,N)$ corresponding to a contraction-free $\TT$-variety $X(\EE)$
and let $\EE'$ be a subdivision of the divisorial fan $\EE$. 
Consider the birational proper equivariant morphism $$f:X(\EE')\rightarrow X(\EE)$$ given by the subdivision $\EE'$. Then we have an isomorphism of constructible complexes of sheaves
$$f_{\star}\,IC_{X(\EE')}\simeq \bigoplus_{\tau\in\germ}\bigoplus_{b\in\ZZ}(i_{\tau})_{\star}IC_{V(\tau)}^{\oplus s_{\tau, b}}[-b]$$
in the derived category $D^b(X(\EE); \QQ)$, where $s_{\tau, b}$ is an $s$-sequence of the subdivision $\EE'$ and $i_{\tau}:V(\tau)\rightarrow X(\EE)$ is the inclusion.
\end{theorem} 
Another way to determine $P_{X(\EE)}(t)$ would be to consider the \emph{stratified multiplicative property}\footnote{We are grateful to the referees for suggesting this reference.} for intersection homology introduced in \cite{CMS08}. This method involves a complex algebraic Whitney stratification on the base curve $Y$ so that the quotient map $\pi: X(\EE)\rightarrow Y$
becomes a stratified submersion. Moreover, we emphasize that the $h$-polynomial $h(Q_{y}(\EE); t^{2})$ is not equal to the Poincar\'e polynomial of the special fiber $\pi^{-1}(y)_{\rm red}$, but to the Poincar\'e polynomial of a certain toric variety of dimension 
$\dim\, X(\EE)$. This latter appearing in our induction process allows us to have a simple expression of $P_{X(\EE)}(t)$. Thus using \cite{CMS08} a substantial work is needed in order to recover Theorem \ref{main1}. 

Our results aim to give a better understanding of the intersection cohomology for complexity-one torus actions on  projective normal varieties. Indeed, one can construct every such a $\TT$-variety from a contraction-free one by contracting certain families of orbits (compare with \cite[Theorem 3.1 (ii)]{AH06}). Hence assuming that these contractions are projective, 
a natural problem would be to describe the decomposition theorem for such maps. Thus using Theorem \ref{main1} this would give a description of the intersection cohomology Betti numbers in this setting.

Let us outline the structure of this article. The first section is separated
into two parts. In the first part, we introduce some notation from the combinatorics of $\TT$-varieties. The second part is a reminder on the intersection cohomology theory. In Section 2, we describe the classical Betti numbers of every 
smooth projective $\TT$-variety of complexity one. Finally, Section 3 is devoted to the proof of Theorems \ref{main1} and \ref{main2}.
\
\newline

{\bf Acknowledgments.} We would like to thank the anonymous referees 
for their valuable comments which allowed to improve the presentation.
The authors also thank Javier Fern\'andez de Bobadilla, Javier Elizondo, Daniel Juteau, Simon Riche, Kari Vilonen, and Geordie Williamson for useful discussions.
The authors benefited from the support of the ERC Consolidator Grant NMST. The second author thanks the Max Planck Institut f\"ur Mathematik Bonn for support.
This research is supported by ERCEA Consolidator Grant 615655 -- NMST and also by the Basque Government through the BERC 2014-2017 program and by Spanish Ministry of Economy and Competitiveness MINECO: BCAM Severo Ochoa excellence accreditation SEV-2013-0323.
This work has been partially supported by ICMAT Severo Ochoa project SEV-2015-0554 (MINECO).

\begin{notation} 
By a \emph{variety} we mean an integral separated scheme of finite type over the field of complex numbers $\CC$. 
An \emph{algebraic torus} (of dimension $n$)
is an affine group scheme over $\CC$ isomorphic to $\G_{m}^{n}$, where $\G_{m}$ is the multiplicative group of the field $\CC$.
If $X$ is a variety, then $\CC[X]$ (resp. $\CC(X)$) denotes the algebra of regular global functions (resp. the field of rational functions) on $X$. We will identify $X$ with its set of $\CC$-points $X(\CC)$ and we will also view
$X$ as a complex manifold. 
Hence the set $X$ is endowed with two topologies: the Zariski topology and the Euclidean topology. 
\end{notation}
\section{Preliminaries}
\subsection{Preliminaries on $\TT$-varieties}
In this section, we recall some basic notions on algebraic torus actions of complexity one (see \cite{AH06, AHS08, Tim08, 
Lan14} for details).
\subsection{}
Let $\TT$ be an algebraic torus with lattice of characters $M$ and lattice of one-parameter subgroups $N$.
Then the duality between $M$ and $N$ extends naturally to a duality
$$M_{\QQ}\times N_{\QQ}\rightarrow \QQ, \,\, (m,v)\mapsto \langle m, v\rangle$$
between the $\QQ$-vector spaces $M_{\QQ} = \QQ\otimes_{\ZZ}M$ and $N_{\QQ} = \QQ\otimes_{\ZZ}N$.
\subsection{}
According to the Sumihiro theorem (see \cite[Section 3, Corollary 2]{Sum74}), every $\TT$-variety is covered by affine $\TT$-stable Zariski open subsets.
Thus we recall first how to describe an affine $\TT$-variety by combinatorial objects in the setting of the complexity one.
\subsection{}
A polyhedral cone $\sigma\subseteq N_{\QQ}$ is said to be \emph{strongly convex} if it contains no lines; this condition is
equivalent to the \emph{dual cone}
$$\sigma^{\vee} = \{m\in M_{\QQ}\,|\, \forall v\in\sigma,\, \langle m, v\rangle \geq 0\}$$
being full dimensional.
Let us fix a strongly convex polyhedral cone $\sigma\subseteq N_{\QQ}$. A \emph{$\sigma$-polyhedron} of $N_{\QQ}$
is a Minkowski sum $Q + \sigma$, where $Q\subseteq N_{\QQ}$ is a polytope (i.e., $Q$ is the convex hull of a non-empty finite subset of $N_{\QQ}$).
Let $Y$ be a smooth curve. A \emph{$\sigma$-polyhedral divisor} $\D$ on the curve $Y$ is a formal sum
$$\D = \sum_{y\in Y}\D_{y}\cdot [y],$$
where every $\D_{y}$ is a $\sigma$-polyhedron and $\D_{y} = \sigma$ for all but finitely many $y\in Y$.
For every $m\in\sigma^{\vee}$ we define a $\QQ$-divisor on $Y$ by letting
$$\D(m) = \sum_{y\in Y}\min_{v\in\D_{y}}\langle m, v\rangle\cdot [y].$$
The curve $Y$ (respectively the cone $\sigma$) is usually called the \emph{locus} (respectively the \emph{tail}) of $\D$.
The $\sigma$-polyhedral divisor $\D$ is called \emph{proper} if $Y$ is affine; or $Y$ is projective and $\D$ verifies the additional properties:
\begin{itemize}
\item[(i)] The \emph{degree} $\deg(\D):=\sum_{y\in Y}\D_{y}$ is strictly contained in $\sigma$.
\item[(ii)] For every $m\in\sigma^{\vee}$ such that $\min_{v\in\deg(\D)}\langle m, v\rangle =0$,
the divisor $\D(dm)$ is principal for some $d\in\ZZ_{>0}$.
\end{itemize}     
We denote by $\PP(Y, \sigma)$ the set of proper $\sigma$-polyhedral divisors on $Y$. If $Y$ is an affine curve, we make the convention that the degree of $\D$
is equal to the empty set. 
\\
 
It is known that for an affine variety $X$ there is a one-to-one correspondence between $\TT$-actions on $X$ and $M$-gradings on $\CC[X]$. The next result (see \cite[Theorems 3.1, 3.4]{AH06}) gives a combinatorial description of $M$-graded algebras corresponding to 
affine $\TT$-varieties of complexity one.
\begin{theorem}
\begin{itemize}
\item[(i)] Let $\sigma\subseteq N_{\QQ}$ be a strongly convex polyhedral cone and let $Y$ be a smooth curve.
If $\D\in \PP(Y, \sigma)$, then the $M$-graded subalgebra
$$A(Y,\D):=\bigoplus_{m\in\sigma^{\vee}\cap M}H^{0}(Y,\mathcal{O}_{Y}(\lfloor \D(m)\rfloor))\otimes \chi^{m}\subseteq \CC(Y)\otimes_{\CC} \CC[\TT],$$
where $\chi^{m}$ is the Laurent monomial corresponding to $m\in M$, defines an affine $\TT$-variety $X(\D) = X(Y,\D)$ of complexity one
with rational quotient $Y$.
\item[(ii)] Conversely, if $X$ is an affine $\TT$-variety of complexity one, then there exist a strongly convex polyhedral cone $\sigma\subseteq N_{\QQ}$,
a smooth curve $Y$, and $\D\in\PP(Y,\sigma)$ such that the $\TT$-variety $X$ is $\TT$-isomorphic to $X(\D)$.
\end{itemize}
\end{theorem} 
We refer to \cite[Section 8]{AH06} and \cite[Section 4]{Lan14} for the functorial properties and the uniqueness problem of this decomposition. In the next paragraph,
we explain the combinatorial description of \cite{AHS08} for (non-necessarily affine) $\TT$-varieties
specialized to the case of the complexity one. 
\subsection{}
\label{divisorial}
Let us fix a smooth curve $Y$. A \emph{divisorial fan} on $(Y,N)$ is a finite set $\EE = \{\D^{i}\,|\, i\in I\}$
with $\D^{i}\in \PP(Y_{i},\sigma_{i})$, where $Y_{i}\subseteq Y$ is a Zariski open dense subset and $\sigma_{i} \subseteq N_{\QQ}$
is a strongly convex polyhedral cone, satisfying the following conditions:
\begin{itemize}
\item[(i)] For all $i,j\in I$ we have 
$$\D^{i}\cap \D^{j} := \sum_{y\in Y_{ij}}(\D^{i}_{y}\cap \D^{j}_{y})\cdot [y]\in \EE,$$
where $Y_{ij} = \{y\in Y_{i}\cap Y_{j}\,|\, \D_{y}^{i}\cap \D_{y}^{j}\neq \emptyset\}$.
\item[(ii)] For all $i,j\in I$ and for every $y\in Y_{ij}$, the polyhedron $\D_{y}^{i}\cap \D_{y}^{j}$
is a common face of $\D_{y}^{i}$ and $\D_{y}^{j}$. 
\item[(iii)] The intersection of the degree of $\D^{i}$ with the tail $\sigma_{ij}$ of $\D^{i}\cap \D^{j}$
is equal to the intersection of the degree of $\D^{j}$ with $\sigma_{ij}$.
\item[(iv)] We have $Y = \bigcup_{i\in I}Y_{i}$.
\end{itemize}
By the preceding conditions, the set $\{\sigma_{i}\,|\,i\in I\}$
generates a fan denoted $\tail(\EE)$.
If $\EE$ is a divisorial fan on $(Y,N)$, then the natural morphisms (cf \cite[Section 5]{AHS08})
$$X(\D^{i})\leftarrow X(\D^{i}\cap \D^{j})\rightarrow X(\D^{j})$$
are $\TT$-invariant Zariski open immersions. The collection of $\TT$-varieties $X(\D^{i})$ can be glued 
in a $\TT$-variety $X(\EE)$ in which the Zariski open subsets $X(\D^{i}\cap \D^{j})$ are identified with
the intersections $X(\D^{i})\cap X(\D^{j})$ (see \cite[Remark 7.4 (ii)]{AHS08} for the fact that the $\CC$-scheme $X(\EE)$ is separated).

Conversely, if $X$ is a $\TT$-variety of complexity one, then there exist a smooth curve $Y$ and 
a divisorial fan $\EE$ on $(Y,N)$ such that the $\TT$-variety $X$ is $\TT$-isomorphic to $X(\EE)$ \cite[Theorem 5.6]{AHS08}. 
Note that the proof of this latter fact uses the Sumihiro theorem.  
\begin{remarque} Certain geometric properties of complexity-one $\TT$-varieties 
can be translated into the language of divisorial fans. For instance, $X(\EE)$ is a complete variety
if and only if $Y$ is a smooth projective curve and 
$$\bigcup_{i\in I}\D_{y}^{i} = N_{\QQ}\text{ for every }y\in Y$$
(c.f \cite[Theorem 7.5]{AHS08}). See also 
\cite[Chapter II]{KKMS73}, \cite[Section 2]{LT16} for a criterion of smoothness and 
\cite[Corollary 3.28]{PS11} for a criterion of projectivity. Moreover, $X(\EE)$ is contraction-free if and only if the locus of each element 
of $\EE$ is an affine curve.
\end{remarque}
Let us fix a divisorial fan $\EE$ on $(Y,N)$ such that $X(\EE)$ is contraction-free.
In the next paragraph, we recall the description of the orbits 
of $\TT$ on $X(\EE)$ in terms of the combinatorial object $\EE = \{\D^{i}\,|\,i\in I\}$ (see \cite[Section 7]{AH06}).
\subsection{}
\label{orbit}
Denote by $\pi: X(\EE)\rightarrow Y$ the quotient morphism where the restriction on
each open subset $X(\D^{i})$ is given by the inclusion $\CC[Y_{i}]\subseteq A(Y_{i},\D^{i})$ (we recall that $Y_{i}$
is the locus of $\D^{i}$). 
Given $y\in Y$, let us explain how to describe the orbits
of the reduced part $\pi^{-1}(y)_{\rm red}$. 
For a vertex $v\in\D^{i}_{y}$ consider the cone
$$\lambda(v) = \{m\in\sigma^{\vee}\,|\, \forall v'\in\D^{i}_{y},\,\langle m, v'-v\rangle \geq 0\}.$$
The irreducible components of $\pi^{-1}(y)_{\rm red}\cap X(\D^{i})$
are identified with the toric varieties $X_{\lambda(v), M_{v}}$ having weight cone $\lambda(v)$ and weight lattice 
$$M_{v} = \{m\in M\,|\, \langle m, v\rangle \in\ZZ\},$$
where $v$ runs the set of vertices of $\D_{y}^{i}$.

Denote by $\face(\EE)_{y}$ the set of faces of the polyhedra $\D_{y}^{i}$.
We have a bijection 
$$F\mapsto O(y,F)$$
between the set $\face(\EE)_{y}$ and the set of $\TT$-orbits of $\pi^{-1}(y)_{\rm red}$; the orbit $O(y, F)$
is of dimension $\codim(F) = \rank(N)-\dim(F)$ and is seen geometrically as 
a common part of the component $X_{\lambda(v), M_{v}}$
for every vertex $v$ of $F$.

Let 
$$\supp(\EE) := \{y\in Y\,|\, \face(\EE)_{y}\neq \tail(\EE)\}.$$
Then we have a natural equivariant identification
$$X(\EE)\setminus \pi^{-1}(\supp(\EE))_{\rm red}\simeq (Y\setminus \supp(\EE))\times X(\tail(\EE)),$$
where $X(\tail(\EE))$ is the toric variety for the torus $\TT$ associated with the fan $\tail(\EE)$. Indeed, the fibers of the quotient 
map $\pi$ over $Y\setminus \supp(\EE)$ are all isomorphic to $X(\tail(\EE))$. Since $\TT$ is a solvable linear algebraic group, the 
$\TT$-isomorphism above is obtained by the cross section theorem in \cite[Section 4, Theorem 10]{Ros56}.
\\

A \emph{prime $\TT$-cycle} (or a \emph{germ}) of $X(\EE)$ is a $\TT$-stable irreducible reduced Zariski closed subset of $X(\EE)$.
A combinatorial description of the prime $\TT$-cycles of $X(\EE)$ is given in \cite[Section 16.4]{Tim11}. We recall this in the 
next paragraph.
\subsection{}
\label{germ}
Let $\D$ be a $\sigma$-polyhedral divisor on a Zariski open dense subset
$Y_{0}\subseteq Y$. 
Assuming that $Y_{0}$ is affine, we define the set $H(\D)$ as the union
of the faces of $\sigma$ and of the pairs of the form $(y, F)$ where $y\in Y_{0}$ and $F$ is a  face of $\D_{y}.$
 
 There exists a bijection between $H(\D)$
and the set of prime $\TT$-cycles of $X(\D)$ given as follows. With the same notation as above, let $v$ be in the relative interior
of $F$ where $(y, F)\in H(\D)$ (resp. $F\in H(\D)$ if $F$ is a face of $\sigma$). Then we define a discrete valuation $\val_{y,v}$ (resp. $\val_{v}$) on $A(Y_{0},\D)$ via the formula   
$$\val_{y,v}(f\otimes\chi^{m}) = \ord_{y}(f) + \langle m, v\rangle \text{ resp. } \val_{v}(f\otimes\chi^{m}) = \langle m, v\rangle,$$
for every homogeneous element $f\otimes\chi^{m}\in A(Y_{0},\D)$. The associated prime $\TT$-cycle is given by the ideal
$$\{\gamma\in A(Y_{0},\D)\setminus\{0\}\,|\, \val_{y,v}(\gamma)>0 \text{ (resp. }\val_{v}(\gamma)>0) \}\cup\{0\}.$$

Let $\EE$ be a divisorial fan on $(Y,N)$ with $X(\EE)$ contraction-free. 
Then we denote by $H(\EE)$ the union of all the sets $H(\D)$ where $\D$ runs through $\EE$. Similarly, the set $H(\EE)$ is in bijection with the set of prime $\TT$-cycles of $X(\EE)$. 
In the sequel, we will denote by $V(\tau)$ the prime $\TT$-cycle associated with $\tau\in H(\EE)$. Note that in this case every
prime $\TT$-cycle of $X(\EE)$ is normal (see \cite[Theorem 7]{Tim00}).

\subsection{Preliminaries on intersection cohomology}
In this section, we fix the notation concerning the intersection cohomology theory that we will use
(see \cite{GM80, GM83}). We note that the cohomology sheaves are particularly sheaves of $\QQ$-vector spaces with respect to the Euclidean
topology. The perversity for the intersection cohomology groups is the middle perversity. We refer the reader to \cite{BBP82} for the theory of perverse sheaves.
\subsection{}
\label{betti1}
For a variety $X$ we denote by $\der$ the bounded derived category of complexes of sheaves of $\QQ$-vectors spaces. 

If $f:X\rightarrow Y$ is a morphism of varieties, then we will denote for simplicity by
$f_{!}, f_{\star}, f^{!}, f^{\star}$ the functors $Rf_{!}, Rf_{\star}, Rf^{!}, Rf^{\star}$ at the level of the derived categories $\der$ and $D^b(Y;\QQ)$.
We denote by $\QQ_X$ the constant sheaf on $X$ with rational coefficients.

By a \emph{local system} on $X$ we mean a locally constant sheaf of $\QQ$-vector spaces (for the Euclidean topology)
having finite dimensional stalks. A local system is \emph{simple} if has no non-trivial local subsystems (i.e., simple in the category 
of local systems) and \emph{semisimple} if it is direct sum of simple local systems.

Let $\L$ be a local system defined on a Euclidean open dense subset of the regular locus
of $X$. We denote by $IC_{X}(\L)$ the \emph{intersection cohomology complex} with coefficient $\L$. This intersection complex is defined as follows.
If $U$ is a smooth open dense subset of $X$ where $\L$ is defined and $\gamma:U\rightarrow X$ is the inclusion, then $IC_{X}(\L)$ 
is the image of the natural morphism
$${}^pH^{0}(\gamma_{!}\L[\dim(X)])\rightarrow {}^pH^{0}(\gamma_{\star}\L[\dim(X)]).$$
Here ${}^pH^{0}$ is the (zero-th) $\mathfrak{t}$-cohomology functor for the natural $\mathfrak{t}$-structure
on the complexes of constructible sheaves in $D^b(Y; \QQ)$, defining the category of perverse sheaves. Note that in the last formula we write by the same letter $\L$ the restriction of $\L$ on $U$.
If $\L$ is constant, then we let $IC_X = IC_{X}(\L)$.
\subsection{}
The rational intersection cohomology 
group $\IH^{\star}(X;\L)$ with coefficient $\L$ is defined via the hypercohomology 
of the intersection complex $IC_{X}(\L)$, i.e.,
$$\IH^{\star}(X;\L) = R^\star\Gamma(X, IC_{X}(\L)).$$ 
Again we denote by $\IH^{\star}(X;\QQ) = \IH^{\star}(X;\L)$ if $\L$ is constant.
The vector spaces $\IH^{\star}(X;\QQ)$ 
satisfy Poincar\'e duality and they are finitely generated (see \cite[Sections 3.2,3.3,4]{GM80}). 
We call $b_{j}(X) = \dim\, \IH^{j}(X;\QQ)$ the \emph{$j$-th intersection cohomology Betti number}. Furthermore, we have the equality
$$b_{j}(X) = 0\text{ if }j<0\text{ or }j>2d,$$ 
see \cite[Section 4.1]{GM83}. We denote by
$$P_{X}(t) = \sum_{j = 0}^{2d}b_{j}(X)\,t^{j}$$
the \emph{Poincar\'e polynomial} of $X$.
\\

In the particular case where $X$ is smooth, rational intersection cohomology
coincides with classical rational cohomology. The following paragraph explains how compute the rational cohomology groups of every smooth projective variety.

\subsection{}
\label{polynomial}
The \emph{Hodge--Deligne polynomial}
of a variety $X$ is defined by the relation  
$$E(X; u,v) = \sum_{p,q = 0}^{d}\sum_{j = 0}^{2d}(-1)^{j}h^{p,q}(H^{j}_{c}(X;\CC))u^{p}v^{q}\in\ZZ[u,v],$$
where $d = \dim(X)$ and $h^{p,q}(H^{j}_{c}(X;\CC))$ is the dimension of the $(p,q)$-type Hodge component in the
$j$-th cohomology group $H^{j}_{c}(X;\CC)$ with compact support. The polynomial $E(\,\star\,; u,v)$
satisfies the following properties:
\begin{itemize}
\item[(i)] If $Z$ is a Zariski closed subset of $X$ and $U = X\setminus Z$, then 
$$E(X; u, v) = E(Z ; u,v) + E(U ; u,v).$$ 
\item[(ii)] If $X_{1}$ and $X_{2}$ are two varieties, then
$$E(X_{1}\times X_{2}; u,v) =E(X_{1}; u, v)\cdot E(X_{2} ; u,v).$$ 
\item[(iii)] If $X$ is smooth and projective, then $P_{X}(t) = E(X; -t, -t)$.
\end{itemize}
For instance, if $Y$ is a smooth projective curve of genus $g$, then
$$E(Y; u,v) = uv - g(u+v) +1.$$ Hence $E(\P^{1}; u,v) = uv +1$, $E(\AA^{1}; u,v) = uv$, 
$E(\G_{m}; u,v) = uv -1$ and so
$$E(\G_{m}^{r}; u,v) = (uv-1)^{r}\text{ for every }r\in\ZZ_{\geq 0}.$$
\begin{remarque}
One can also define the intersection cohomology $E$-polynomial of a singular
quasi-projective variety in a similar
way as the Hodge--Deligne polynomial for classical cohomology.
We refer the reader to \cite[Section 3]{BB96} for some properties.
It would be interesting to describe this polynomial in the setting of $\TT$-varieties.
\end{remarque}

In the proof of Proposition \ref{decomp}, we will use the following technical lemma. It is a consequence of results in \cite[Section 4.2]{BBP82}.
We thank Simon Riche for communicating the idea of the proof of this lemma.
\begin{lemma}
\label{pullback}
Let $f:X\rightarrow Y$ be an  \'etale morphism between two varieties $X$ and $Y$. Then $f^{\star}IC_{Y}$ identifies
with the intersection complex $IC_{X}$. 
\end{lemma}
\begin{proof}
Let $d = \dim(X)$.
If $U$ is a smooth open dense subset of $Y$ and $\gamma:U\rightarrow Y$ is the inclusion, then recall that $IC_{Y}$
is defined as the image of the natural morphism
$${}^pH^{0}(\gamma_{!}\QQ_{U}[d])\rightarrow {}^pH^{0}(\gamma_{\star}\QQ_{U}[d]).$$
Since the functor $f^{\star}$ is exact (see \cite[p.108-109]{BBP82}),
the complex $f^{\star}IC_{Y}$ is the image of the induced morphism
$${}^pH^{0}(f^{\star}\gamma_{!}\QQ_{U}[d])\rightarrow {}^pH^{0}(f^{\star}\gamma_{\star}\QQ_{U}[d]).$$
Let us consider the Cartesian diagram
$$\xymatrix{
    f^{-1}(U) \ar[d]^\beta \ar[r]^f & U  \ar[d]^\gamma\\
    X \ar[r]^{f}  & Y  
 }$$
where $\beta: f^{-1}(U)\rightarrow X$ is the inclusion. Note that $f^{-1}(U)$ is again a smooth dense open subset of $X$.
By base change we have the isomorphisms
$$f^{\star}\gamma_{!}\QQ_{U}[d]\simeq \beta_{!}f^{\star}\QQ_{U}[d]\simeq \beta_{!}\QQ_{f^{-1}(U)}[d],$$
in the derived category $\der$.
Since $f$ is \'etale, we have $f^{\star} = f^{!}$ (see \cite[Section 4.2.4]{BBP82}), hence by base change
$f^{\star}\gamma_{\star}\QQ_{U}[d] \simeq \beta_{\star}\QQ_{f^{-1}(U)}[d]$. Thus we obtain that $f^{\star}IC_{Y}$
is the image of
$${}^pH^{0}(\beta_{!}\QQ_{f^{-1}(U)}[d])\rightarrow {}^pH^{0}(\beta_{\star}\QQ_{f^{-1}(U)}[d])$$
which coincides with $IC_{X}$.
\end{proof}
From now on to simplify the notation we make the convention that the intersection cohomology complex $IC_{X}$ of an algebraic variety $X$ is the one defined before 
shifted by $\dim\, X$. In particular, $IC_{X} = \QQ_{X}[\dim\, X]$ provided that $X$ 
is smooth.
We enunciate the decomposition theorem of Beilinson--Bernstein--Deligne--Gabber
which allows us to describe the topology of singular proper algebraic maps (see \cite[Theorem 6.25]{BBP82}).
\begin{theorem}
Let $f:X\rightarrow Z$ be a proper (algebraic) morphism between varieties $X$ and $Z$. Then there exists a finite family $(Z_{\alpha}, \L_{\alpha}, d_{\alpha})$
where for every index $\alpha$, $Z_{\alpha}\subseteq Z$ is a smooth irreducible (algebraic) subvariety which is locally closed (for the Zariski topology),
$\L_{\alpha}$ is a semisimple local system on $Z_{\alpha}$ and $d_{\alpha}\in\ZZ$ such that we have an isomorphism
$$f_{\star}\,IC_{X}\simeq \bigoplus_{\alpha}(i_{\alpha})_{\star}\,IC_{\bar{Z}_{\alpha}}(\L_{\alpha})[-d_{\alpha}]$$
in the derived category $D^b(Z; \QQ)$, 
where $\bar{Z}_{\alpha}$ is the Zariski closure of $Z_{\alpha}$ in $Z$ and $i_{\alpha}:\bar{Z}_{\alpha}\rightarrow Z$
is the inclusion. This implies that for every Euclidean open subset $U\subseteq Z$ and for every $j\in\ZZ$, we have an isomorphism
of $\QQ$-vector spaces 
$$IH^{j}(f^{-1}(U);\QQ)\simeq \bigoplus_{\alpha}IH^{j-l_{\alpha}}(U\cap \bar{Z}_{\alpha}; \L_{\alpha}),$$
where $l_{\alpha} = \dim(X)-\dim(Z_{\alpha})+d_{\alpha}$. 
\end{theorem}
In \cite{CMM15} the preceding theorem was made explicit in the case where $X$ and $Z$ are toric varieties
and $f$ is a proper toric map. For convenience we recall this result.
\begin{theorem}
\label{decomptoric}
Let $X,Z$ be toric varieties for the torus $\TT$. Let $\Sigma_{Z}$ be the fan defining $Z$.
For every $\tau\in\Sigma_{Z}$, denote by $V(\tau)$ the corresponding orbit closure. 
Consider a birational proper toric map $f:X\rightarrow Z$  given by a subdivision of fans. Then we have an isomorphism
$$f_{\star}\,IC_{X}\simeq \bigoplus_{\tau\in\Sigma_{Z}}\bigoplus_{b\in\ZZ}(i_{\tau})_{\star}\,IC_{V(\tau)}^{\oplus s_{\tau, b}}[-b]$$
in the derived category $D^b(Z; \QQ)$,
where $i_{\tau}:V(\tau)\rightarrow Z$ is the inclusion.
This implies that for every Euclidean open subset $U\subseteq Z$ and every $j\in\ZZ$,
we have an isomorphism of $\QQ$-vector spaces 
$$IH^{j}(f^{-1}(U);\QQ)\simeq \bigoplus_{\tau\in\Sigma_{Z}}\bigoplus_{b\in\ZZ}IH^{j-l_{\tau, b}}(U\cap V(\tau); \QQ)^{\oplus s_{\tau, b}},$$
where $l_{\tau, b} = b+\dim(X)-\dim(V(\tau))$. The sequence $s_{\tau, b}$ of natural numbers depends on the fan subdivision and satisfies the following conditions.
\begin{itemize}
\item[(i)] $s_{\tau, b} = s_{\tau, -b}$ for all $\tau\in\Sigma_{Z}$ and $b\in\ZZ$.
\item[(ii)] If $f$ is a projective morphism, then $s_{\tau, b}\geq s_{\tau, b+2l}$ for all $\tau\in\Sigma_{Z}$ and all $b,l\in\ZZ_{\geq 0}$.
\item[(iii)] $s_{\tau, b} = 0$ if $l_{\tau, b}$ is odd. 
\item[(iv)] $s_{0,0} = 1$ and $s_{0,b} = 0$ for every $b\in\ZZ\setminus\{0\}$.
\end{itemize}
\end{theorem}
Note that assertion (iv) in Theorem \ref{decomptoric} is a consequence of \cite[Remark 7.1, Theorems 7.2,7.4]{CMM15}.
The following terminology will be useful for the sequel.
\begin{definition}
\label{sequence}
With the same notation as in \ref{decomptoric}, let us denote by $\Sigma_{X}$ the fan defining the toric variety $X$.
The sequence of coefficients $s_{\tau, b}$ ($\tau$ and $b$ run respectively $\Sigma_{Z}$ and $\ZZ$) involving
in Theorem \ref{decomptoric} will be called an \emph{$s$-sequence} of the subdivision $\Sigma_{X}$ of $\Sigma_{Z}$.
\end{definition}

Let us end this section by recalling the notion of $h$-vectors which describes the 
intersection cohomology Betti numbers of projective toric varieties (see \cite[Theorem 3.1]{Sta87} and \cite{DL91, Fie91}).
\subsection{}
\label{hpol}
We will consider here the empty set $\emptyset$ as a common face of every polytope.
Let us define two polynomials $h(\,\star\,, t), g(\,\star\,, t)\in \ZZ[t]$ by induction \cite[Section 2]{Sta87},
where the first entry $\star$ depends on a polytope. 
Let $g(\emptyset, t) = 1$ and $\dim(\emptyset) =-1$.
Assume that $g(\,\star\,, t)$ is defined for every polytope of dimension $<d$. Then for a polytope $Q$
of dimension $d$, we let 
$$h(Q,t) = \sum_{\Lambda\subseteq Q}g(\Lambda,t) (t-1)^{d-1-\dim(\Lambda)},$$
where the sum runs through the set of proper faces of $Q$ (included the empty set $\emptyset$). Let us denote by $h_{k}(Q)$ the $k$-th coefficient 
of $h(Q,t)$. One defines $g(Q,t)$ by the relation
$$g(Q,t) = \sum_{k = 0}^{\lfloor d/2\rfloor}(h_{k}(Q) - h_{k-1}(Q))t^{k}\text{ with } h_{-1}(Q) = 0.$$
The polynomial $h(Q,t)$ is called the \emph{$h$-polynomial} of $Q$ and $(h_{0}(Q),\ldots, h_{d}(Q))$ is the \emph{$h$-vector} of $Q$. 
More generally, we will let $h_{j}(Q) = 0$ for any $j\in\ZZ_{<0}$. 
We refer to \cite{Fie91} for a geometric meaning of $h(\,\star\,, t)$ and $g(\,\star\,, t)$.  
\begin{theorem}
\label{hvector}
Let $X$ be a projective toric variety defined by a fan $\Sigma_{X}$. Let
us fix a polytope $Q\subseteq M_{\QQ}$ such that $\Sigma_{X}$ is the normal fan of $Q$. Then we have
the equality $P_{X}(t) = h(Q,t^{2})$.
\end{theorem}

\section{Betti numbers of complexity-one smooth $\TT$-varieties}
\label{betti}
The purpose of this section is to give an explicit description of the classical Betti
numbers of every smooth projective $\TT$-variety of complexity one. Before stating our results,
let us introduce some notation.
\subsection{}
\label{polytopes}
Let $\EE$ be a divisorial fan on $(Y,N)$ such that $X(\EE)$ is contraction-free. Let $\D\in \EE$ and let $y\in Y$ be a point belonging to the locus of $\D$. The \emph{associated  Cayley cone} of $\D$ at the point $y$ is the subset $\hc(\D)_{y}$ in $N_{\QQ}\oplus\QQ$ generated by $(\sigma\times\{0\})\cup(\D_{y}\times\{1\})$. 
We denote by $\hc^{-}(\D)$ the cone generated by the subset
$$(\sigma\times\{0\})\cup(\sigma\times\{-1\})\subseteq N_{\QQ}\oplus\QQ,$$
where $\sigma$ is the tail of $\D$. 
We denote by $\EE_{y}^{+}$ resp. $\EE_{y}$ the fan generated by 
$$\{\hc(\D)_{y}\,|\,\D\in \EE \text{ with }y\text{ in the locus of } \D\}$$ 
$$\text{ resp. }\{\hc^{-}(\D), \hc(\D)_{y}\,|\,\D\in \EE\text{ with }y\text{ in the locus of } \D\}.$$
Note that in the last expression $\D$ is no longer fixed.
If $X(\EE)$ is projective, then $\EE_{y}$ is a complete fan
which is a normal fan of a polytope $Q_{y}(\EE)$ (compare with \cite[Corollary 3.28]{PS11}).
The same holds for the fan $\tail(\EE)$; we denote by $Q(\EE)$ a polytope such that $\tail(\EE)$ is its normal fan.
\\

Let us illustrate the preceding paragraph by an example where we consider a toric variety with an action
of a codimension-one subtorus.
\begin{example}
\label{example1}
Consider the divisorial fan
$$\EE = \{\D^{i,0},\D^{i,\infty}\,|\, i\in I\}$$
on $(\P^{1}, N)$, where $\D^{i,0}\in \PP(\P^{1}\setminus \{0\}, \sigma_{i})$ and $\D^{i,\infty}\in\PP(\P^{1}\setminus \{\infty\}, \sigma_{i})$. 
Assume that $\D^{i,0}_{y} = \D^{i,\infty}_{y} = \sigma_{i}$ for all $i\in I$ and $y\in \P^{1}\setminus \{0\}$.
Indentifying
$\P^{1}\setminus \{0\}$ resp. $\P^{1}\setminus\{\infty\}$ with $\spec(\CC[t^{-1}])$ resp. $\spec(\CC[t])$ we have the equality
$$A(\P^{1}\setminus\{\infty\}, \D^{i, \infty}) = \CC[\hc(\D^{i, \infty})_{0}^{\vee}\cap (M\oplus\ZZ)],$$
for every $i\in I$, where the right-hand side is the semigroup algebra of $$\hc(\D^{i, \infty})_{0}^{\vee}\cap (M\oplus\ZZ).$$
Similarly, we have
$$A(\P^{1}\setminus\{0\}, \D^{i, 0}) = \CC[\hc^{-}(\D^{i, 0})^{\vee}\cap (M\oplus\ZZ)].$$
Consequently, $X(\EE)$ is the toric variety for the torus $\big$ associated with the fan $\EE_{0}$.
\end{example}
%The description of the toric variety of Example \ref{example1} allows us to establish relations
%between $h$-polynomials of polytopes: 
%\begin{lemme}
%\label{lemma1}
%Let $\EE$ be a divisorial fan on $(\P^{1}, N)$ as in Example \ref{example1} and assume that $X(\EE)$ is smooth projective. Then we have the equality
%$$h(Q_{0}(\EE); t^{2}) = t^{2}\, h(Q(\EE); t^{2}) + \sum_{F\in \face(\EE)_{0}}(t^{2}-1)^{\codim(F)}.$$ 
%\end{lemme}
%\begin{proof}
%First of all, by Theorem \ref{hvector} we know that $h(Q_{0}(\EE); t^{2}) = P_{X(\EE)}(t)$. Let us compute $P_{X(\EE)}(t)$
%by using the method of $E$-polynomials (see \ref{polynomial}). Removing the special fiber at the origin of the quotient map
%$\pi: X(\EE)\rightarrow \P^{1}$ gives
%$$E(X(\EE) ; u, v) = E(\AA^{1}\times X(\tail(\EE)) ; u, v) + E(\pi^{-1}(0)_{\rm red} ; u , v).$$
%Using the orbit decomposition for $\pi^{-1}(0)_{\rm red}$ (see \ref{orbit}) we obtain that 
%$$E(X(\EE) ; u, v) = uv\, E(X(\tail(\EE)) ; u, v)  + \sum_{F\in \face(\EE)_{0}}(uv-1)^{\codim(F)}.$$
%The result follows from the substitution $u = v = -t$.
%\end{proof}
The next result gives in particular an explicit formula of the Poincar\'e polynomial
of each smooth projective contraction-free $\TT$-variety of complexity one.
We consider this result in the general setting of quotient singularities 
since it will be needed for the proof of Theorem \ref{main}. 
\begin{proposition}
\label{proposition}
Let $X(\EE)$ be a projective contraction-free $\TT$-variety of complexity one
corresponding to a divisorial fan $\EE$ on $(Y, N)$.
Denote by $g$ the genus of the curve $Y$ and by $r$ the cardinality of the finite set $\supp(\EE)$.
Assume that every cone of $\EE_{y}$ is generated by a subset of a basis of $N_{\QQ}\oplus\QQ$ for all $y\in Y$.
Then we have the equality
$$P_{X(\EE)}(t) =  ((1-r)t^{2} + 2gt + 1- r)h(Q(\EE); t^{2}) + \sum_{y\in \supp(\EE)}h(Q_{y}(\EE); t^{2}).$$
\end{proposition}
\begin{proof}
We first compute the Poincar\'e polynomial $P_{X(\EE)}(t)$ with the assumption that $X(\EE)$ is smooth.
Using again the method with $E$-polynomials we have
$$E(X(\EE); u, v) = E((Y\setminus \pi^{-1}(\supp(\EE))_{\rm red})\times X(\tail(\EE)); u, v) + 
E(\pi^{-1}(\supp(\EE))_{\rm red}; u , v)$$
$$ = (uv - (u + v)g + 1 - r)E(X(\tail(\EE)); u, v ) + \sum_{y\in \supp(\EE)}\sum_{F\in\face(\EE)_{y}}(uv - 1)^{\codim(F)}.$$
By substituting $u = v = -t$, it follows that
$$P_{X(\EE)}(t) = \sum_{y\in \supp(\EE)}\left[t^{2}h(Q(\EE); t^{2}) + \sum_{F\in\face(\EE)_{y}}(t^{2} - 1)^{\codim(F)}\right]$$
$$+ ((1-r)t^{2} + 2gt + 1-r)h(Q(\EE); t^{2}).$$ 
According to \cite[Chapter II]{KKMS73} the smoothness assumption on $X(\EE)$ implies that the toric variety $X(\EE_{y})$ associated with the fan
$\EE_{y}$ is smooth for every $y\in Y$. By using the description of the fiber of quotient map in \cite[Proposition 7.10]{AH06} on the right-hand side 
($\EE_{0}$ of Example \ref{example1} is replaced by $\EE_{y}$), we obtain the required formula.

In the general case, one can realize $X(\EE)$ as the quotient $X(\hat{\EE})/G,$
where $X(\hat{\EE})$ is smooth and $G$ is a finite abelian group by changing the ambient lattice $N$. Hence $X(\EE)$ is rationally smooth \cite[Proposition A1 iii)]{Bri99} and so 
$$H^{j}(X(\EE); \QQ) = IH^{j}(X(\EE); \QQ)\text{ for any $j\in\ZZ$.}$$ 
To obtain the last equalities, we may adapt \cite[Proposition 8.2.21]{HTT08} for the intersection cohomology with rational coefficients. Since the quotient map by $G$ induces isomorphisms between
rational cohomology groups of $X(\EE)$ and $X(\hat{\EE})$ we conclude that $P_{X(\EE)}(t) = P_{X(\hat{\EE})}(t)$. 
Finally, we can make the choice to take $Q(\EE)$ and $Q_{y}(\EE)$ so that 
$$Q(\EE) = Q(\hat{\EE})\text{ and }Q_{y}(\EE) = Q_{y}(\hat{\EE})$$ for every $y\in Y$.
This finishes the proof of the proposition. 
\end{proof}
\subsection{}
Actually, the $h$-vectors of simplicial polytopes are simple to describe (compare Theorem \ref{hvector} with \cite[Corollary 3.9]{CMM15}) and
the formula of Proposition \ref{proposition} can be expressed as the relation 
$$P_{X(\EE)}(t) = ((1-r)t^{2}+2gt+1-r)\left(\sum_{i=0}^{d-1}f_{i}(Q(\EE))(t^{2}-1)^{i}\right)$$ 
$$+\sum_{y\in\supp(\EE)}\sum_{i=0}^{d-1}f_{i}(Q_{y}(\EE))(t^{2}-1)^{i},$$
where $d = \dim(X(\EE))$ and for any simplicial polytope $Q$, we denote by $f_{i}(Q)$ the number of faces of dimension $i$.
\\

In the next paragraphs, we complete the computation of Betti numbers of smooth projective $\TT$-varieties of complexity one.
Note that according to the Luna slice theorem \cite[Chapter III]{Lun73} every smooth non-contraction-free $\TT$-variety of complexity one is automatically 
rational (see also \cite[Section 2]{LT16}). Hence we may restrict ourselves to divisorial fans on the projective line. 

We want to obtain a formula similar
to \ref{proposition} where orbits in special and non-special fibers are separated.
More precisely, the idea is to use the result \cite[Theorem 10.1]{AH06}
describing in particular the rational quotient of a non-contraction-free
$\TT$-variety of complexity one. 

%\begin{remarque}
%\label{git}
%Note that an affine $\TT$-variety $Z = X(\D)$ of complexity one has its quasi-fan of GIT-chambers (see \cite[Section 5]{AH06} or \cite[Section 2]{BH06}) generated by exactly one cone $\omega_Z$, which
%is the dual of the tail of $\D$. 
%A proof of this fact can be obtained from results in \cite[Section 3]{Lan14}. For instance, assume that the locus of $\D$ is a complete curve $Y$. Let us
%denote by $Y_u$ the GIT-quotient associated with $u\in \omega_Z\cap M$.
%If $\D(u)$ is big (i.e $\deg\, \D(m)>0$), then $Y_u=Y$. Otherwise, $Y_u$ is a
%point. We will use this fact later on.
%\end{remarque}
\subsection{}
\label{Phi}
Let $\EE$ be a divisorial fan on $(\P^{1},N)$ such that $X(\EE)$ is complete.
Let us introduce specific notation. We denote by $\bfan(\EE)$ the subset of elements $\tau\in\tail(\EE)$ with the property that 
there exists $\D\in\EE$ such that $\tau\cap \deg(\D)= \emptyset$.

We will consider the set $$\Phi_{\star} = \{(y, F)\, |y\in\supp(\EE)\text{ and } F\in\face(\EE)_{y}\}.$$ 
Moreover, we denote by $\Phi$ the quotient set $\Phi_{\star}/\mathscr{R}$, where the equivalence relation $\mathscr{R}$
is defined as follows. We have $(y, F)\,\RR\, (y',F')$ if and only if $F$ and $F'$ have the same tail cone and this tail does not belong to $\bfan(\EE)$.
The class of $(y,F)$ modulo $\mathscr{R}$ will be denoted by $[y,F]$. Finally, let us remark that the map
$$\codim:\Phi\rightarrow \ZZ_{\geq 0},\,\,[y,F]\mapsto\codim(F)$$
is well defined and so for every integer $i\in\ZZ_{\geq 0}$ we can let
$$\Phi_{i} = \{[y,F]\in\Phi\,|\, \codim(F) = i\}.$$ 
\begin{proposition}
\label{proposition2}
Let $\EE$ be a divisorial fan on $(\P^{1}, N)$ such that $X(\EE)$ is smooth and projective. 
Let $r$ be the cardinality of $\supp(\EE)$.
With the same notation 
as in \ref{Phi}, we have the equality
$$P_{X(\EE)}(t) = (t^{2}+1-r)\left(\sum_{i = 0}^{d-1}c_{i}(t^{2}-1)^{i}\right)+ \sum_{i = 0}^{d-1}\sharp \Phi_{i}\cdot (t^{2}-1)^{i},$$
where $d = \dim(X(\EE))$ and $c_{i}$ is the number of cones of $\bfan(\EE)$ of codimension $i$.
\end{proposition}
\begin{proof}
Let us consider the equivariant map $q:\tilde{X}\rightarrow X(\EE)$ obtained by gluing the natural maps
$$\spec_{Y_{0}}\left(\bigoplus_{m\in\sigma^{\vee}\cap M}\mathcal{O}_{Y_{0}}(\lfloor \D(m)\rfloor)\otimes\chi^{m}\right)\rightarrow X(\D),$$
where $\D\in\EE$ and $Y_{0}$ resp. $\sigma$ is the locus resp. the tail of $\D$. The map $q$ being proper and birational, it follows that $X(\EE) = q(\tilde{X})$.
Let $O(\tau)\subseteq X(\tail(\EE))$ be the $\TT$-orbit corresponding to $\tau$.
Note that the set 
$$Z_{1} = (\P^{1}\setminus\supp(\EE))\times \bigcup_{\tau\in\bfan(\EE)}O(\tau)$$
identifies with a subset of $\tilde{X}$. We denote by $Z_{2}$ its complement. Moreover,
we observe by using \cite[Theorem 3.1]{AH06} that $q(x_{1})\neq q(x_{2})$ for all $x_{1}\in Z_{1}$ and $x_{2}\in Z_{2}$, and therefore we have the decomposition
$$E(X(\EE); u, v) = E(q(Z_{1}); u, v) + E(q(Z_{2}); u, v),$$
where the second part corresponds to $\sum_{[y,F]\in \Phi}E(O(y,F); u, v).$
Hence by substituting $u = v = -t$, we conclude.
\end{proof}

\section{Poincar\'e polynomials in the singular case}
\subsection{A description of the decomposition theorem}
In this section, we investigate the decomposition theorem
for equivariant proper maps between contraction-free $\TT$-varieties of complexity one induced by subdivisions of divisorial fans
(see Proposition \ref{decomp}). Let us introduce some notation.
\subsection{}
\label{tau}
Let $\EE$ be a divisorial fan on $(Y,N)$ such that $X(\EE)$ is contraction-free. Recall that $$\pi:X(\EE)\rightarrow Y$$ stands for the quotient map. We denote by $\germ$
the finite subset of $H(\EE)$ (see \ref{germ}) which consists of elements of the form:
\begin{itemize}
\item[(1)] $\tau\in\tail(\EE)$ (horizontal type)
\item[(2)] pair $(y, F)$ (vertical type)
where $F$ is a face of $\D_{y}$ for some $\D\in\EE$ and some $y\in\supp(\EE)$. 
\end{itemize}
The set of elements of type (1) resp. (2) will be denoted by $\germh$ resp. $\germv$. 

Note that the prime $\TT$-cycle $V(\tau)$ associated with
$\tau\in\germh$ is the subvariety obtained as the Zariski closure of $(Y\setminus\supp(\EE))\times \tilde{V}(\tau)$ in $X(\EE)$,
where $\tilde{V}(\tau)$ is the orbit closure of $X(\tail(\EE))$ corresponding to $\tau$ and $(Y\setminus \supp(\EE))\times X(\tail(\EE))$
is identified with the complement of $\pi^{-1}(\supp(\EE))_{\rm red}$ in $X(\EE)$. 
To every $\tau\in\germh$ one can attach a divisorial fan $\EE(\tau)$ describing the prime $\TT$-cycle $V(\tau)$.

Moreover, every prime $\TT$-cycle $V(\tau)$ with $\tau\in\germv$ is contained
in $\pi^{-1}(\supp(\EE))_{\rm red}$. If $X(\EE)$ is projective, then $V(\tau)$
corresponds to a projective toric variety; we denote by $Q_{\tau}$ a polytope such that the fan defining
$V(\tau)$ is the normal fan of $Q_{\tau}$. The notation $\EE(\tau)$ resp. 
$Q_{\tau}$ will apply in Section \ref{calculation}.
\begin{remarque}
The prime $\TT$-cycles contained in each fiber $\pi^{-1}(y)$ where $y\not\in\supp(\EE)$ correspond
to elements of $H(\EE)\setminus SH(\EE)$.  
\end{remarque}
The next paragraph introduces notions of subdivision of divisorial fans and $s$-sequences
similar to the toric case.
\begin{definition}
\label{subf}
With the same notation as in \ref{tau} for the divisorial fan $\EE$, we say that a divisorial fan $\EE'$ on $(Y, N)$
is a \emph{subdivision} of $\EE$ if for every $y\in Y$ it verifies:
\begin{itemize}
\item[(i)] For all $\D'\in\EE'$ there exists $\D\in \EE$ such that $\hc(\D')_{y}\subseteq\hc(\D)_{y}$.
\item[(ii)] For all $\D\in\EE$ there exist $\D^{1},\ldots, \D^{m}\in\EE'$ with the same loci as $\D$ 
such that $\hc(\D)_{y} = \bigcup_{i\in\{1,\ldots, m\}}\hc(\D^{i})_{y}$. 
\end{itemize}
Clearly a subdivision $\EE'$ of $\EE$ naturally induces an equivariant birational morphism $$X(\EE')\rightarrow X(\EE)$$ which is proper (see \cite[Theorem 12.13]{Tim11}). 
Note that $X(\EE')$ is also contraction-free and that $(\EE')^{+}_{y}$ is a subdivision fan of $\EE^{+}_{y}$
for every $y\in Y$. Similarly, $\tail(\EE')$ is a subdivision of $\tail(\EE)$. 
An \emph{$s$-sequence} of the subdivision $\EE'$ of $\EE$ is a sequence $s_{\tau, b}$ of natural numbers where $\tau\in\germ$ and $b\in\ZZ$
such that:
\begin{itemize}
\item[(a)] The sequence $s_{b,\tau}$ (for $\tau\in\germh = \tail(\EE)$) is a usual $s$-sequence (see Definition \ref{sequence}) of the fan subdivision $\tail(\EE')$ of 
$\tail(\EE)$.
\item[(b)] For every $y\in\supp(\EE)$, the sequence $s_{b,\tau}$ (for $\tau\in\EE^{+}_{y}$) is a usual $s$-sequence of the fan subdivision $(\EE')^{+}_{y}$ of $\EE^{+}_{y}$. Here we identify $\tau$ with its Cayley cone.
\end{itemize} 
\end{definition}
The following is the main result of this section. To prove this result, we observe that a contraction-free $\TT$-variety of complexity one is covered by explicit \'etale open subsets of toric varieties. This idea was developed in \cite[Chapter II]{KKMS73} for toroidal embeddings. Thus, we will reduce the proof to the statement \ref{decomptoric}.   
\begin{proposition}
\label{decomp}

Let $\EE$ be a divisorial fan on $(Y,N)$ corresponding to a contraction-free $\TT$-variety $X(\EE)$
and let $\EE'$ be a subdivision of the divisorial fan $\EE$. 
Consider the birational proper equivariant morphism $$f:X(\EE')\rightarrow X(\EE)$$ given by the subdivision $\EE'$. Then we have an isomorphism
$$f_{\star}\,IC_{X(\EE')}\simeq \bigoplus_{\tau\in\germ}\bigoplus_{b\in\ZZ}(i_{\tau})_{\star}IC_{V(\tau)}^{\oplus s_{\tau, b}}[-b]$$
in the derived category $D^b(X(\EE); \QQ)$, where $s_{\tau, b}$ is an $s$-sequence of the subdivision $\EE'$ and $i_{\tau}:V(\tau)\rightarrow X(\EE)$ is the inclusion.
This implies that for every Euclidean open subset $U\subseteq X(\EE)$ and every $j\in\ZZ$,
we have an isomorphism of $\QQ$-vector spaces 
$$IH^{j}(f^{-1}(U);\QQ)\simeq \bigoplus_{\tau\in\germ}\bigoplus_{b\in\ZZ}IH^{j-l_{\tau, b}}(U\cap V(\tau); \QQ)^{\oplus s_{\tau, b}},$$
where $l_{\tau, b} = b+\dim(X(\EE))-\dim(V(\tau))$. 
\end{proposition}
\begin{proof}
Let $\D\in\EE$. Denote by $Y_{0}$ the locus of $\D$ and fix a point $z\in Y_{0}$. 
Since $Y_{0}$ is smooth,
there exist two Zariski open neighborhoods
$$z\in U_{z}\subseteq Y_{0}\text{ and } 0\in V_{0}\subseteq \AA^{1}$$ 
and a surjective \'etale morphism $\varphi: U_{z}\rightarrow V_{0}$ sending $z$ to $0$ with the condition that $\varphi^{-1}(0) = \{z\}$.
Let $\D^{z}$ be the polyhedral divisor on $\AA^{1}$ defined by the conditions $\D_{0}^{z} = \D_{z}$
and $\D^{z}_{y} = \sigma$ for any $y\neq 0$, where $\sigma$ is the tail of $\D$.

Shrinking $U_{z}$ and $V_{0}$ if necessary, we may suppose that $\D^{z}_{y} = \sigma$ for any $y\in V_{0}\setminus \{0\}$
and $\D_{y'} = \sigma$ for any $y'\in U_{z}\setminus\{z\}$. Let us denote 
$$\D_{|U_{z}} = \sum_{y\in U_{z}}\D_{y}\cdot [y]\text{  and  }\D^{z}_{|V_{0}} = \sum_{y\in V_{0}}\D_{y}^{z}\cdot [y].$$ 
Then the map $\varphi$ induces an \'etale morphism $X(\D_{|U_{z}})\rightarrow X(\D^{z}_{|V_{0}})$. Indeed,
the preceding map may view as the natural projection
$$X(\D_{|U_{z}})\simeq X(\D^{z}_{|V_{0}})\times_{V_{0}}U_{z}\rightarrow X(\D^{z}_{|V_{0}}).$$
We also remark that $X(\D^{z}_{|V_{0}})$ is a Zariski open subset of the affine toric $\TT\times\G_{m}$-variety $X(\hc(\D)_{z})$ associated with the cone $\hc(\D)_{z}$. 

Let $\EE'(\D_{|U_{z}})$ resp. $\EE'(\D^{z}_{|V_{0}})$ be the divisorial fan of polyhedral divisors $\D'_{|U_{z}}$ resp. $(\D')^{z}_{|V_{0}}$, where 
$\D'$ runs $\EE'$ such that $\D$ and $\D'$ have the same locus and
$\hc(\D')_{y}\subseteq \hc(\D)_{y}$ for any $y$ belonging to this common locus. Then $\varphi$ also induces a commutative diagram:
$$\xymatrix{
    X(\EE'(\D_{|U_{z}})) \ar[d]^j \ar[r]^f & X(\D_{|U_{z}})  \ar[d]^i\\
    X(\EE'(\D^{z}_{|V_{0}})) \ar[r]^{f'}  & X(\D^{z}_{|V_{0}})  
 }$$
The horizontal maps $f$ and $f'$ are given by subdivisions of divisorial fans and the vertical maps $i$ and $j$ are \'etale morphisms. 
Furthermore, we have the equalities
$$f^{-1}(X(\D_{|U_{z}})) = X(\EE'(\D_{|U_{z}}))\text{ and } (f')^{-1}(X(\D^{z}_{|V_{0}})) = X(\EE'(\D^{z}_{|V_{0}})).$$
The map $f'$ is the restriction of a toric map given by a fan subdivision. By Theorem \ref{decomptoric} we have 
$$f'_{\star}\,IC_{X(\EE'(\D^{z}_{|V_{0}}))}\simeq  \bigoplus_{\tau\in SH(\D^{z}_{|V_{0}})}\bigoplus_{b\in\ZZ}(i_{\tau})_{\star}IC_{V'(\tau)}^{\oplus s_{\tau, b}}[-b],$$
where $\tau \mapsto V'(\tau)$ is the parametrization of the prime $\TT$-cycles in $X(\D^{z}_{|V_{0}})$. Since $i,j$ are smooth of relative dimension $0$, we have $i^{\star} = i^{!}$ and $j^{\star} = j^{!}$ (see \cite[Section 4.2.4]{BBP82}). By construction of the maps $i,j$ the diagram above is Cartesian. Hence
by base change and the equalities $i^{-1}(V'(\tau)) = V(\tau)\cap X(\D_{|U_{z}})$ it follows that 
$$f_{\star} IC_{X(\EE'(\D_{|U_{z}}))} \simeq f_{\star} j^{\star} IC_{X(\EE'(\D^{z}_{|V_{0}}))}\simeq i^{\star} f'_{\star}\,IC_{X(\EE'(\D^{z}_{|V_{0}}))}$$ 
$$\simeq  \bigoplus_{\tau\in SH(\D^{z}_{|V_{0}})}\bigoplus_{b\in\ZZ}i^{\star}(i_{\tau})_{\star}IC_{V'(\tau)}^{\oplus s_{\tau, b}}[-b].$$
Note that the first isomorphism is given by Lemma \ref{pullback}. 
Using the Cartesian diagram
$$\xymatrix{
    V(\tau)\cap X(\D_{|U_{z}}) \ar[d]^i \ar[r]^{i_{\tau}} & X(\D_{|U_{z}})  \ar[d]^i\\
    V'(\tau) \ar[r]^{i_{\tau}}  & X(\D^{z}_{|V_{0}})  
 }$$
and Lemma \ref{pullback} we have by base change the isomorphisms
$$(f_{\star}\,IC_{X(\EE')})_{|X(\D_{|U_{z}})}\simeq  \bigoplus_{\tau\in\germ}\bigoplus_{b\in\ZZ}(i_{\tau})_{\star}IC_{V(\tau)}^{\oplus s_{\tau, b}}[-b]_{|X(\D_{|U_{z}})},$$
where $s_{\tau, b}$ is an $s$-sequence of the divisorial fan subdivision $\EE'$.
Since $X(\EE)$ is contraction-free we can cover $X(\EE)$ by open subsets 
$$X(\D^{1}_{|U_{z_{1}}}),\ldots, X(\D^{r}_{|U_{z_{r}}}),
\text{ where }\D^{i}\in \EE, z_{i}\in Y$$ and the open dense subset $U_{z_{i}}\subseteq Y$ sastifying $U_{z_{i}}\cap\supp(\D) \subseteq \{z_{i}\}.$
By the argument above we may choose these open subsets so that we have isomorphisms
\begin{eqnarray}
(f_{\star}\,IC_{X(\EE')})_{|X(\D_{|U_{z_{i}}})}\simeq  \bigoplus_{\tau\in\germ}\bigoplus_{b\in\ZZ}(i_{\tau})_{\star}IC_{V(\tau)}^{\oplus s_{\tau, b}}[-b]_{|X(\D_{|U_{z_{i}}})}\text{ for all }
i = 1,\ldots r.
\end{eqnarray}
Using the uniqueness of the decomposition into simple perverse sheaves, \cite[Remark 2.5.3]{Dim04} and (1), we conclude that 
$$f_{\star}\,IC_{X(\EE')}\simeq \bigoplus_{\tau\in\germ}\bigoplus_{b\in\ZZ}(i_{\tau})_{\star}IC_{V(\tau)}^{\oplus s_{\tau, b}}[-b],$$
where $s_{\tau, b}$ is an $s$-sequence of the divisorial fan subdivision $\EE'$. 
Note that the local systems in the right-hand side of the previous relation are trivial since by (1) 
they are trivial on a Zariski open dense subset (see \cite[Remark 2.5.3]{Dim04}).
This finishes the proof of the proposition.   
\end{proof}

\subsection{Calculation of the intersection cohomology}
The aim of this section is to extend the statement of Proposition \ref{proposition} to the singular case.
Let us start with some preliminary results.
\label{calculation}
\begin{lemma}
\label{subdivision}
Let $X(\EE)$ be a projective contraction-free $\TT$-variety of complexity one
corresponding to a divisorial fan $\EE$ on $(Y, N)$. Then there exists a subdivision $\EE'$ of the divisorial
fan $\EE$ such that for any $y\in Y$ each cone of $\EE'_{y}$ is generated by a subset of a basis of $N_{\QQ}\oplus\QQ$ and the natural morphism $X(\EE')\rightarrow X(\EE)$
is projective.
\end{lemma}
\begin{proof}
First of all, we consider a star subdivision (see \cite[Section 11.1]{CLS11}) for every one-dimensional cone of $\tail(\EE)$. This induces
a subdivision $\EE''$ of $\EE$. Secondly, to each $y\in\supp(\EE'')$ we consider a star subdivision of $(\EE'')^{+}_{y}$
for every one-dimensional cone of $(\EE'')^{+}_{y}\setminus\tail(\EE'')$. This last process induces again a subdivision $\EE'$ of $\EE$ (see \cite[Section 5.3]{LPR16} for more details). We conclude by \cite[Corollary 3.28]{PS11} and the proof
of \cite[Proposition 11.1.7]{CLS11} that we may find the subdivision $\EE'$ as required.
\end{proof}
As in Section \ref{betti}, the next result establishes relations between $h$-polynomials of polytopes starting with
Example \ref{example1}.
\begin{lemma}
\label{formula1}
Let $\EE,\EE'$ be divisorial fans on $(\P^{1}, N)$ as in Example \ref{example1}. Assume that $\EE'_{0}$
is a subdivision of $\EE_{0}$ such that $X(\EE')$ is projective and each cone of $\EE'_{0}$ is generated by a 
subset of a basis of $N_{\QQ}\oplus\QQ$. Let $d = \dim(X(\EE))$ and $d_{\tau} = \dim(V(\tau))$
for $\tau\in H(\EE)$. Then we have the formula
$$h_{j+d}(Q_{0}(\EE)) = h_{j+d}(Q_{0}(\EE')) - \sum_{\tau\in \germh,\,\tau\neq 0}\sum_{b\in\ZZ}s_{\tau, b}\cdot h_{j+d_{\tau}-b}(Q_{0}(\EE(\tau)))$$
$$-\sum_{\tau\in \germv,\,\tau\in\EE^{+}_{0}}\sum_{b\in\ZZ}s_{\tau, b}\cdot h_{j+d_{\tau}-b}(Q_{\tau})$$
for every $j\in\ZZ_{\geq 0}$, 
where $s_{\tau, b}$ is any $s$-sequence of the subdivision $\EE'$ (see the notation in \ref{tau}).
\end{lemma}
\begin{proof}
Applying Theorem \ref{decomptoric} for the (toric) proper map $X(\EE')\rightarrow X(\EE)$ induced by the subdivision $\EE'_{0}$ of $\EE_{0}$
we obtain
$$b_{j+d}(X(\EE')) = \sum_{\tau\in \germ}\sum_{b\in\ZZ}s_{\tau, b}\cdot b_{j+d_{\tau}-b}(X(\EE(\tau)))$$
for every $j\in\ZZ$. We conclude by using Theorem \ref{hvector} and the fact that $s_{0,0} = 1$ and $s_{0,b} = 0$
for all $b\in\ZZ\setminus\{0\}$. 
\end{proof}
The reader may remark that the formula of the next result depends only of the tail fans $\tail(\EE)$ and $\tail(\EE')$. Furthermore, the proof of this result is a straightforward consequence of Theorem \ref{decomptoric}
and so we omitted it.
\begin{lemma}
\label{formula2}
Let $\EE$ be a divisorial fan on $(Y,N)$ such that $X(\EE)$ is contraction-free and projective.
Let $\EE'$ be a subdivision of $\EE$ as in Lemma \ref{subdivision}.
We recall that $Q(\EE)$ denotes a polytope such that $\tail(\EE)$
is a normal fan of $Q(\EE)$. Then we have the equality
$$h_{j+d}(Q(\EE)) = h_{j+d}(Q(\EE')) - \sum_{\tau\in\germh,\,\tau\neq 0}\sum_{b\in\ZZ} s_{\tau,b}\cdot h_{j+d_{\tau}-b}(Q(\EE(\tau)))$$
for every $j\in\ZZ$, where $s_{\tau, b}$ is any $s$-sequence of the subdivision $\EE'$.
\end{lemma}
The following is our main result.
\begin{theorem}
\label{main}
Let $X(\EE)$ be a singular projective contraction-free $\TT$-variety of complexity one
corresponding to a divisorial fan $\EE$ on $(Y, N)$.
Denote by $g$ the genus of the curve $Y$ and by $r$ the cardinal of the set $\supp(\EE)$.
Then we have the equality
$$P_{X(\EE)}(t) =  ((1-r)t^{2} + 2gt + 1- r)h(Q(\EE); t^{2}) + \sum_{y\in \supp(\EE)}h(Q_{y}(\EE); t^{2}).$$
In particular, if $X(\EE)$ is rational, then $IH^{2j+1}(X(\EE); \QQ) = 0$ for every $j\in\ZZ$.
\end{theorem}
\begin{proof}
We show the result by induction on the dimension $d$ of $X(\EE)$. In the initial step $d = 1$ we have $X(\EE) = Y$. So 
$P_{X(\EE)}(t) =  t^{2} + 2gt + 1$ and the result holds in this step.   
Assume that the result holds in dimension $<d$ where $d\geq 2$. 
Let us consider a subdivision $\EE'$ of $\EE$ as in Lemma \ref{subdivision}.
Using Proposition \ref{decomp}, we
can write for every $j\in\ZZ$:
$$b_{j+d}(X(\EE)) = b_{j+d}(X(\EE')) - \sum_{\tau\in\germ,\,\tau\neq 0}\sum_{b\in\ZZ}s_{\tau, b}\cdot b_{j+d_{\tau}-b}(X(\EE(\tau)))$$
$$ = b_{j+d}(X(\EE')) - \sum_{\tau\in\germh,\,\tau\neq 0}\sum_{b\in\ZZ}s_{\tau, b}\cdot b_{j+d_{\tau}-b}(X(\EE(\tau)))$$
$$-\sum_{y\in\supp(\EE)}\sum_{\tau\in\germv,\,\tau\in\EE^{+}_{y}}\sum_{b\in\ZZ}s_{\tau, b}\cdot b_{i+d_{\tau}-b}(X(\EE(\tau))).$$ 
We apply the induction process for the prime $\TT$-cycles $V(\tau) = X(\EE(\tau))\subsetneq X(\EE)$ so that
$$b_{j+d}(X(\EE)) = \sum_{y\in\supp(\EE)}h_{j+d}(Q_{y}(\EE'))$$
$$+(1-r)\, h_{j+d-2}(Q(\EE'))+2g\, h_{j+d-1}(Q(\EE'))+(1-r)\, h_{j+d}(Q(\EE'))$$
$$-\sum_{\tau\in\germh,\,\tau\neq 0}\sum_{b\in\ZZ}s_{\tau, b}\cdot [\sum_{y\in\supp(\EE)}h_{j+d_{\tau}-b}(Q_{y}(\EE(\tau)))$$
$$+(1-r)\, h_{j+d_{\tau}-b-2}(Q(\EE(\tau)))+2g\, h_{j+d_{\tau}-b-1}(Q(\EE(\tau)))+(1-r)\, h_{j+d_{\tau}-b}(Q(\EE(\tau)))]$$
$$-\sum_{y\in\supp(\EE)}\sum_{\tau\in\germv,\,\tau\in\EE^{+}_{y}}\sum_{b\in\ZZ}s_{\tau, b}\cdot h_{j + d_{\tau}-b}(Q_{\tau}).$$
We observe that since $\supp(\EE(\tau))\subseteq \supp(\EE)$ for any $\tau\in\germ$ (see \ref{tau}) we may substitute the sets $\supp(\EE(\tau))$
by the set $\supp(\EE)$ in the second and third sums of the right-hand side of the preceding equality.  
Using Lemmas \ref{formula1} and \ref{formula2} we have
$$b_{j+d}(X(\EE)) = (1-r)\cdot[h_{j+d-2}(Q(\EE'))-\sum_{\tau\in\germh,\,\tau\neq 0}\sum_{b\in\ZZ}s_{\tau,b}\cdot h_{j+d_{\tau}-b-2}(Q(\EE(\tau)))]$$
$$+2g\cdot[h_{j+d-1}(Q(\EE'))-\sum_{\tau\in\germh,\,\tau\neq 0}\sum_{b\in\ZZ}s_{\tau,b}\cdot h_{j+d_{\tau}-b-1}(Q(\EE(\tau)))]$$
$$+(1-r)\cdot[h_{j+d}(Q(\EE'))-\sum_{\tau\in\germh,\,\tau\neq 0}\sum_{b\in\ZZ}s_{\tau,b}\cdot h_{j+d_{\tau}-b}(Q(\EE(\tau)))]$$
$$+\sum_{y\in\supp(\EE)}[h_{j+d}(Q_{y}(\EE'))-\sum_{\tau\in\germ,\,\tau\neq 0}\sum_{b\in\ZZ}s_{\tau, b}\cdot h_{j+d_{\tau}-b}(Q_{y}(\EE(\tau)))$$
$$-\sum_{\tau\in\germv,\,\tau\in\EE^{+}_{y}}\sum_{b\in\ZZ}s_{\tau, b}\cdot h_{j+d_{\tau}-b}(Q_{\tau})]$$
$$= (1-r)h_{j+d-2}(Q(\EE)) + 2g h_{j+d-1}(Q(\EE)) +(1-r) h_{j+d}(Q(\EE)) + \sum_{y\in\supp(\EE)}h_{j+d}(Q_{y}(\EE)).$$
This gives the formula for the Poincar\'e polynomial $P_{X(\EE)}(t)$. Finally, for the last claim we apply the formula for $g = 0$.   
\end{proof}

\end{document}